\documentclass[12pt]{amsart}
\usepackage{amsmath,amssymb,amsbsy,amsfonts,latexsym,amsopn,amstext,
                                               amsxtra,euscript,amscd,bm}
\usepackage{comment}

\usepackage{units}
\usepackage{float}
\restylefloat{table}
\usepackage[english]{babel}                    
\usepackage{url}
\usepackage[colorlinks,linkcolor=blue,anchorcolor=blue,citecolor=blue]{hyperref}
\usepackage{color}
\usepackage[utf8]{inputenc}

\usepackage{listings}
\usepackage{xcolor}

\usepackage{blindtext}

\definecolor{codegreen}{rgb}{0,0.6,0}
\definecolor{codegray}{rgb}{0.5,0.5,0.5}
\definecolor{codepurple}{rgb}{0.58,0,0.82}
\definecolor{backcolour}{rgb}{0.95,0.95,0.92}

\lstdefinestyle{mystyle}{
    backgroundcolor=\color{backcolour},   
    commentstyle=\color{codegreen},
    keywordstyle=\color{magenta},
    numberstyle=\tiny\color{codegray},
    stringstyle=\color{codepurple},
    basicstyle=\ttfamily\footnotesize,
    breakatwhitespace=false,         
    breaklines=true,                 
    captionpos=b,                    
    keepspaces=true,                 
    numbers=left,                    
    numbersep=5pt,                  
    showspaces=false,                
    showstringspaces=false,
    showtabs=false,                  
    tabsize=2
}

\usepackage[margin=1.3in]{geometry} 

\begin{document}

\newcommand{\Mod}[1]{\ (\mathrm{mod}\ #1)}

\newtheorem*{note}{Note}
\newtheorem{problem}{Problem}
\newtheorem{theorem}{Theorem}
\newtheorem{lemma}[theorem]{Lemma}
\newtheorem{claim}[theorem]{Claim}
\newtheorem{corollary}[theorem]{Corollary}
\newtheorem{prop}[theorem]{Proposition}
\newtheorem{definition}{Definition}
\newtheorem{question}[theorem]{Question}
\newtheorem{conjecture}{Conjecture}
\def\cA{{\mathcal A}}
\def\cB{{\mathcal B}}
\def\cC{{\mathcal C}}
\def\cD{{\mathcal D}}
\def\cE{{\mathcal E}}
\def\cF{{\mathcal F}}
\def\cG{{\mathcal G}}
\def\cH{{\mathcal H}}
\def\cI{{\mathcal I}}
\def\cJ{{\mathcal J}}
\def\cK{{\mathcal K}}
\def\cL{{\mathcal L}}
\def\cM{{\mathcal M}}
\def\cN{{\mathcal N}}
\def\cO{{\mathcal O}}
\def\cP{{\mathcal P}}
\def\cQ{{\mathcal Q}}
\def\cR{{\mathcal R}}
\def\cS{{\mathcal S}}
\def\cT{{\mathcal T}}
\def\cU{{\mathcal U}}
\def\cV{{\mathcal V}}
\def\cW{{\mathcal W}}
\def\cX{{\mathcal X}}
\def\cY{{\mathcal Y}}
\def\cZ{{\mathcal Z}}

\def\A{{\mathbb A}}
\def\B{{\mathbb B}}
\def\C{{\mathbb C}}
\def\D{{\mathbb D}}
\def\E{{\mathbb E}}
\def\F{{\mathbb F}}
\def\G{{\mathbb G}}
\def\I{{\mathbb I}}
\def\J{{\mathbb J}}
\def\K{{\mathbb K}}
\def\L{{\mathbb L}}
\def\M{{\mathbb M}}
\def\N{{\mathbb N}}
\def\O{{\mathbb O}}
\def\P{{\mathbb P}}
\def\Q{{\mathbb Q}}
\def\R{{\mathbb R}}
\def\S{{\mathbb S}}
\def\T{{\mathbb T}}
\def\U{{\mathbb U}}
\def\V{{\mathbb V}}
\def\W{{\mathbb W}}
\def\X{{\mathbb X}}
\def\Y{{\mathbb Y}}
\def\Z{{\mathbb Z}}

\def\ep{{\mathbf{e}}_p}
\def\em{{\mathbf{e}}_m}
\def\eq{{\mathbf{e}}_q}

\def\scr{\scriptstyle}
\def\\{\cr}
\def\({\left(}
\def\){\right)}
\def\[{\left[}
\def\]{\right]}
\def\<{\langle}
\def\>{\rangle}
\def\fl#1{\left\lfloor#1\right\rfloor}
\def\rf#1{\left\lceil#1\right\rceil}
\def\le{\leqslant}
\def\ge{\geqslant}
\def\eps{\varepsilon}
\def\mand{\qquad\mbox{and}\qquad}

\def\sssum{\mathop{\sum\ \sum\ \sum}}
\def\ssum{\mathop{\sum\, \sum}}
\def\ssumw{\mathop{\sum\qquad \sum}}

\def\vec#1{\mathbf{#1}}
\def\inv#1{\overline{#1}}
\def\num#1{\mathrm{num}(#1)}
\def\dist{\mathrm{dist}}

\def\fA{{\mathfrak A}}
\def\fB{{\mathfrak B}}
\def\fC{{\mathfrak C}}
\def\fU{{\mathfrak U}}
\def\fV{{\mathfrak V}}

\newcommand{\bflambda}{{\boldsymbol{\lambda}}}
\newcommand{\bfxi}{{\boldsymbol{\xi}}}
\newcommand{\bfrho}{{\boldsymbol{\rho}}}
\newcommand{\bfnu}{{\boldsymbol{\nu}}}

\def\GL{\mathrm{GL}}
\def\SL{\mathrm{SL}}

\def\Hba{\overline{\cH}_{a,m}}
\def\Hta{\widetilde{\cH}_{a,m}}
\def\Hb1{\overline{\cH}_{m}}
\def\Ht1{\widetilde{\cH}_{m}}

\def\flp#1{{\left\langle#1\right\rangle}_p}
\def\flm#1{{\left\langle#1\right\rangle}_m}
\def\dmod#1#2{\left\|#1\right\|_{#2}}
\def\dmodq#1{\left\|#1\right\|_q}

\def\Zm{\Z/m\Z}

\def\Err{{\mathbf{E}}}

\newcommand{\comm}[1]{\marginpar{%
\vskip-\baselineskip 
\raggedright\footnotesize
\itshape\hrule\smallskip#1\par\smallskip\hrule}}

\def\xxx{\vskip5pt\hrule\vskip5pt}

\newenvironment{nouppercase}{%
  \let\uppercase\relax%
  \renewcommand{\uppercasenonmath}[1]{}}{}
  

\title{Medium-sized values for the Prime Number Theorem for primes in arithmetic progression}
\author{Matteo Bordignon\\ University of New South Wales Canberra, School of Science \\ m.bordignon@student.unsw.edu.au }

\date{\today
}


\begin{abstract}
We give two improved explicit versions of the prime number theorem for primes in arithmetic progression: the first isolating the contribution of the Siegel zero and the second completely explicit, where the improvement is for medium-sized values. This will give an improved explicit Bombieri--Vinogradov like result for non-exceptional moduli.
 \end{abstract}
\begin{nouppercase}
\maketitle
\end{nouppercase}
\section{Introduction}
The prime number theorem for primes in arithmetic progression (PNTPAP) states 
\begin{equation*}
\psi(x;q,a)=\frac{x}{\varphi(q)}+o_q(x),
\end{equation*}
and the strength of the result lies in the explicit value of $o(x)$ and how this depends on the range of $q$. The best known result is the one with the error term due to Siegel--Walfisz that is uniform for $q \le (\log x)^A$, for any $A\ge 0$, but this result can not be made explicit since the proof is ineffective.
The classical explicit versions of the PNTPAP are the following. The first is by McCurley that in \cite{McCurley1} obtains an explicit result for non-exceptional moduli and in \cite{McCurley2} focuses on the case where $q=3$. Improving \cite{McCurley2}, Ramaré and Rumley prove in \cite{R-R} explicit results for $q\le 72$ and other small moduli. 
A result for large moduli is obtained by Liu and Wang in \cite{Liu-Wang}, they prove, for $q \le \log^6 x$, a version of the PNTPAP with an explicit error term of size $\frac{x}{\log^{13} x}$. Dusart, in \cite{Dusart}, obtains an explicit error term that is of size $o \left(\frac{x}{\log^A x} \right) $, for any $A>0$, note that while this result improves \cite{Liu-Wang} for large $x$, it is worse for medium-sized values, that is somewhere in the range $10^2 \le \log x \le 10^6$. Dusart also improves the result in \cite{R-R} for $q=3$. Yamada in \cite{Yamada} (unpublished) proves a generalized version of the result in \cite{Liu-Wang}, where for multiple small to medium-sized $A$ and for $a\le \log^A x$, he isolates the contribution of the Siegel zero and obtains an error term of size $\frac{x}{\log^{A-2} x}$. Note that this result is better than the one in \cite{Dusart} for medium-sized values, aside for the non-explicit contribution of the Siegel zero. Yamada also uses this result, joint with \cite{Amir}, to obtain an explicit version of a Bombieri--Vinogradov style theorem for non-exceptional moduli. The last explicit version of the PNTPAP is the one by Bennett et al.\ in \cite{Bennett}, here they improve the previous results for $3 \le p \le 4\cdot 10^5$.\newline
In this paper we will focus on a version of the PNTPAP for medium-sized $x$. In doing this we will draw inspiration from \cite{Yamada} and we will first obtain an improved explicit version with isolated the contribution of the Siegel zero, see Theorem \ref{theo:1}. To obtain this result it is fundamental to obtain, with $\chi$ a Dirichlet character modulo $q$,  the best error term in
\begin{equation}
\label{eq:in}
\psi(x,\chi)=\sum_{\rho \in z(\chi),|\gamma|\le T} \frac{x^{\rho}}{\rho}+R(x) \frac{x \log x}{T},
\end{equation}
see Section \ref{sec:pr} for the definition of $z(\chi)$.
Setting $T=\log^A x$, with $0<A\ll 1$, Yamada, drawing inspiration from \cite{Liu-Wang}, proved explicitly that $R(x)\ll \log x$. In Lemma \ref{lemma:CW1} we improve this result by proving an explicit version of the result by Goldsot in \cite{Goldston}, to obtain $R(x)\ll \log \log x$. In doing this we reshape the proof of Goldston to obtain a better explicit upper bound. Note that all the results we obtain are as general as possible to make them useful in different ranges of $x$ and for different choices of $A$. We then use Theorem \ref{theo:1}, together with the explicit bound on the Siegel zeroes in \cite{Bordignon1} and \cite{Bordignon2}, to obtain a completely explicit version of the PNTPAP, see Theorem \ref{theo:2}, that improves the previous results for medium-sized values. We conclude the paper using Theorem \ref{theo:1} improving the Bombieri--Vinogradov style theorem for non-exceptional moduli in \cite{Yamada}, the result we prove is also more general. \newline 
We now better introduce the three main results.
We start with a result on zeroes on Dirichlet $L$-functions, namely Theorem 1.1 and 1.3 of \cite{Kadiri}.
\begin{theorem}
\label{theo:kadiri}
Define $\prod (s,q)=\prod_{\chi \pmod q} L(s,\chi)$, $R_0=6.3970$ and $R_1=2.0452$. Then the function $\prod(s,q)$ has at most one zero $\rho=\beta+i \gamma$, in the region $\beta \ge 1-1/R_0\log \max \lbrace q,q|\gamma|\rbrace$. Such  zero is called Siegel zero and if it exists, then it must be real, simple and must correspond to a non-principal real character $\chi \pmod q$. Moreover, for any given $Q_1$, between all the zeroes with $q \le Q_1$ there is at most one such zero with $\beta \ge 1-1/2R_1 \log Q_1$.
\end{theorem}

We are interested in an intermediate explicit result to the PNTPAP, that isolates the possible contribute due to the Siegel zero. In this paper we aim to improve Theorem 1.1 in \cite{Yamada}, we will do so in the following result.
\begin{theorem}
\label{theo:1}
Let $\alpha_1$, $\alpha_2 \in \mathbb{R}^{+}$, $Y_0=\log\log X_0$, $ Y_0 \ge \max \left\{\frac{11 \log 10}{\alpha_1+\alpha_2+3}, 2\right\}$ and $C(\alpha_1,\alpha_2,Y_0)$ be a constant given in \eqref{eq:C}. Let $q \le \log^{\alpha_1} x$. Let $E_0=1$ if $\beta_0$, the possible Siegel zero modulo $q$, exists and $E_0=0$ otherwise. If $(a,q)=1$ and $x \ge X_0 $, then
\begin{equation}
\label{eq:1}
-1+x^{-1}\sum_{\chi \pmod{q}}|\psi(x, \chi)| <\frac{C(\alpha_1,\alpha_2,Y_0)}{\log^{\alpha_2}x}+E_0\frac{x^ {\beta_0-1}}{\beta_0}
\end{equation}
and 
\begin{equation}
\label{eq:2}
\frac{\varphi(q)}{x}\left| \psi(x;q,a)-\frac{x}{\varphi(q)}\right|<\frac{C(\alpha_1,\alpha_2,Y_0)}{\log^{\alpha_2}x}+E_0\frac{x^ {\beta_0-1}}{\beta_0}.
\end{equation}
\end{theorem}
See Table \ref{tab:cn} for some explicit upper bounds for $C(\alpha_1,\alpha_2,Y_0)$.
Note again that, for medium-sized $x$, we obtain an improvement of size $\log \log x / \log x$ on Yamada's result in \cite{Yamada}. 
From Theorem \ref{theo:1}, using the results \cite{Bordignon1} and \cite{Bordignon2} to control the size of the Siegel zeroes, we obtain a completely explicit version of the prime number theorem for primes in arithmetic progression.
\begin{theorem}
\label{theo:app}
Let $\alpha_1$, $\alpha_2 \in \mathbb{R}^{+}$, $\alpha_1 < 2$, $Y_0=\log\log X_0$, $ Y_0 \ge \max \left\{\frac{11 \log 10}{\alpha_1+\alpha_2+3},2 \right\}$, $q  \le \log^{\alpha_1} x$ and  
\begin{equation*}
C_1(\alpha_1,\alpha_2,X_0)= 2\exp \left(-\frac{100 (\log x)^{1-\alpha_1/2}}{ (\alpha_1\log \log x)^2}\right)\log^{\alpha_2}x+C(\alpha_1,\alpha_2,Y_0).
\end{equation*}
We have
\begin{equation*}
\left| \psi(x;q,a)-\frac{x}{\varphi(q)}\right|<\frac{C_1(\alpha_1,\alpha_2,X_0)x}{\varphi(q) \log^{\alpha_2}x}.
\end{equation*}
\end{theorem}
See Table \ref{tab:cn1} for some explicit upper bounds for $C_1(\alpha_1,\alpha_2,X_0)$.
\begin{table}[H]
    \begin{tabular}{ | l | l | l | l | }
    \hline
     $Y_0$ & $\alpha_1$ & $\alpha_2$ & $C_1 $ \\  
    \hline
     $6.2$ & $1$ & $1$ & $1.7 \cdot 10^{-4}$ \\
     \hline
     $ 7$ & $ 1$ & $ 2 $ & $3.4\cdot 10^{-5} $  \\
     \hline
     $7.6 $ & $1 $ & $3 $ & $1.1\cdot 10{-5} $ \\
     \hline
     $8.8 $ & $1 $ & $ 10$ & $1.3 \cdot 19^{-6} $ \\
     \hline
     $7.6 $ & $1.2 $ & $1 $ & $1.1\cdot 10^{-5} $ \\
     \hline
     $8.1 $ & $1.2 $ & $2 $ & $4.4 \cdot 10^{-5}$ \\
     \hline
    \end{tabular}
    \quad
    \begin{tabular}{ | l | l | l | l |  }
    \hline
     $Y_0$ & $\alpha_1$ & $\alpha_2$ & $C_1  $ \\  
    \hline
    $10.5 $ & $1.2 $ & $3 $ & $ 5.5\cdot 10^{-5}$ \\
     \hline
     $13.2 $ & $1.2 $ & $7 $ & $ 9.5\cdot 10^{-6}$ \\
     \hline
     $ 11$ & $1.3 $ & $ 1$ & $ 1.3\cdot 10^{-5}$ \\
     \hline
     $13.6 $ & $1.3 $ & $2 $ & $7.9 \cdot 10^{-5}$ \\
     \hline
     $ 15$ & $1.3 $ & $3 $ & $1.2 \cdot 10^{-2} $ \\
     \hline
     $12 $ & $1.4 $ & $1 $ & $0.77$ \\
     \hline
    \end{tabular}
    \quad

\caption{Upper bound for $C_1(x,\alpha_1,\alpha_2,X_0)$}  
\label{tab:cn1}   
\end{table}
Note that the bound $\alpha_1 <2$ is needed to bound effectively the size of the Siegel zero. It would also be interesting to choose a constant upper bound for $q$. If we need an upper bound independent or $\varphi(q)$ we can easily obtain it using Theorem 15 in \cite{R-S}.
For medium-sized $x$ this improves the results in \cite{Bennett}, see for example Theorems 1.1, 1.2, 1.3 and Lemma 6.10.\newline
Now, for a given $Q_1$, we call a modulus $q_0 \le Q_1$ to be exceptional up to $Q_1$ if $\prod (s,q_0)$ has a zero with $\beta \ge 1-1/2R_1 \log Q_1$, with $R_1$ defined in Theorem \ref{theo:kadiri}. We then use Theorem \ref{theo:1} and Theorem 1.2 in \cite{Amir} to improve Theorem 1.4 in \cite{Yamada}, while taking care in giving a more general version of the result. This result is an explicit Bombieri--Vinogradov like theorem where the sum is restricted to non-exceptional moduli.
\begin{theorem}
\label{theo:2}
Let $A\in \mathbb{R}^{+}$ be such that $A>3$,
\begin{equation*}
c_0= \frac{2^{\frac{13}{2}}\left( 2+\log (\log 2/\log (4/3))\right)}{9\pi (\log 2)^2} \left(\frac{1}{3}+\frac{3}{2\log 2} \right)\sqrt{\frac{\psi(113)}{113}}
\end{equation*}
and
\begin{equation*}
c_1=\prod_p\left( 1+\frac{1}{p(p-1)}\right).
\end{equation*}
Let $Q=\frac{x^{1/2}}{\log^A x}$ and $1\le Q_1 \le \log^Ax$. Let $q_0$ denote the exceptional modulus up to $Q_1$ if it exists. If $\log \log x \ge \max \{7, \frac{11\log 10}{2A} \}$, then we have the inequality 
\begin{equation*}
\sum_{q \le Q, q_0 \nmid q} \max_{a \Mod{q}} \left|\psi(x;q,a)-\frac{x}{\varphi(q)} \right|<
\end{equation*}
\begin{equation*}
\sqrt{x}+\frac{c_1c_02x}{\log^{A-\frac{9}{2}}x}+\frac{2c_1c_0x\log^{\frac{9}{2}}x}{Q_1}
+\frac{c_1^2xC(A,A-3,X_0)(1+A\log\log x)\log x}{2\log^{A-4}x}+ E(x,A),
\end{equation*}
with $C( A,A-3,X_0)$ from Theorem \ref{theo:1}, and
\begin{align*}
&E(x,A)=(\log \left(x\log^{A} x \right) )^2\frac{\sqrt{x}}{4\log^{A-2} x}+\frac{c_1^2(1+A\log\log x)\log x}{2}\cdot \\ &\cdot\left( \frac{x^{1-\frac{1}{2AR_1 \log \log x}}}{1-\frac{1}{2AR_1 \log \log x}}+34x(\log x)^{1.52}\exp \left(-0.81\sqrt{\log x} \right) +\frac{\log q}{\log 2}\right)+ \\ & +\frac{c_0c_1\log^{\frac{9}{2}} x}{2} \left( 4\sqrt{x}\log^A x +18\frac{x^{\frac{11}{12}}}{\log ^{\frac{A}{2}}x}+5x^{\frac{5}{6}}+\frac{5}{2}x^{\frac{5}{6}}\log x \right).
\end{align*}
\end{theorem}
This paper is structured as follows. In Section \ref{sec.intro} we introduce useful bounds on the zeta and $L$ functions. In Section \ref{sec:pr} we prove the explicit version of \eqref{eq:in} and in Section \ref{sec:theo1} we conclude the Proof of Theorem \ref{theo:1}. In Section \ref{sec:theo2} we prove Theorem \ref{theo:2}.
\section{Some useful bounds}
\label{sec.intro}
We start introducing Corollary 2.1 in \cite{Bennett1} that, improving \cite{Trudgian}, gives a bound on the value of $N(T, \chi)$, the number of zeroes up to $T$ of $L(s,\chi)$.
\begin{lemma}
\label{lemma:N}
Let $\chi$ be a character with conductor $q>1$. If $T\ge 5/7$, then
\begin{equation*}
\left|N(T, \chi)-\frac{T}{\pi}\log \frac{qT}{2 \pi e} \right|\le r_1(q,T):=\min \left\{0.247 \log qT+ 6.894,0.298 \log qT+ 4.358 \right\}.
\end{equation*}
\end{lemma}
To bound the possible Siegel zeroes we will use the results in \cite{Bordignon1} and \cite{Bordignon2}, that can be stated as follows.
\begin{lemma}
\label{lemma:B}
Let $\beta_0$ be a Siegel zero and $q$ the modulus of the corresponding character $\chi$, then we have
\begin{equation*}
\frac{100}{\sqrt{q}\log^2 q} \le \beta_0 \le 1-\frac{100}{\sqrt{q}\log^2 q}.
\end{equation*}
\end{lemma}
Moreover, for $1< b \le 1.3$, by equation (1.17) in \cite{R-S1} we have
\begin{equation}
\label{eq:z'/z}
-\frac{\zeta'}{\zeta}(b)< \frac{1}{b-1} - C + 0.1877 (b - 1),
\end{equation}
with $C$ the Euler--Mascheroni constant. 
Also, taken $b>1$ and with $t \ge 1126$, by partial summation and Theorem 4 and 14 in \cite{R-S}, we obtain
\begin{align*}
\nonumber &\sum_{\frac{t}{2}<n<t-1/2}\Lambda (n)=\psi(t-1/2)-\psi\left(\frac{t}{2}\right) \le r_2(t):=\frac{t}{2}\left(1 +\frac{1}{\log (t-1/2)}-\frac{1}{2\log t/2}\right)\\ &+\sqrt{t-1/2}\left(1+\frac{1}{\log (t-1/2)}\right)-0.98\sqrt{\frac{t}{2}}+3(t-1/2)^{1/3}-1/2\left(1+\frac{1}{2\log (t-1/2)}\right)
\end{align*}
and this gives
\begin{equation}
\label{eq:Lambdan}
\sum_{\frac{t}{2}<n<t-1/2}\Lambda (n) n^{-b}\le \left(\frac{2}{t}\right)^br_2(t).
\end{equation}

\section{Preliminary results}
\label{sec:pr}
Letting $\rho =\beta +i \gamma$, we define
\begin{equation*}
z(\chi)=\sharp \lbrace \rho \mid \rho \neq 0, \beta >1/2, L(\rho,\chi)=0 \rbrace.
\end{equation*}
that is fundamental in estimating the error term in the prime number theorem in arithmetic progression.
In this section we aim to prove the following fundamental result, that is an improved version of Theorem 8 in \cite{Liu-Wang} and Lemma 2.1 in \cite{Yamada}. 
\begin{lemma}
\label{lemma:main}
Let $\chi$ be a Dirichlet character modulo $q$ and $T=\log^{\alpha} x$, with $\alpha=\alpha_1+\alpha_2+3$, $\alpha_1, \alpha_2 \in \mathbb{R}^{+}$ and $q \le \log^{\alpha_1} x$.
Assuming $x > \exp \exp C$, with $C\ge 2$, we have
\begin{equation}
\label{eq:0.}
|\psi(x,\chi)-\delta(\chi)x|\le  \sum_{\rho \in z(\chi),|\gamma|\le T} \frac{x^{\beta}}{|\rho|}+R(C,\alpha_2, \alpha_1) \frac{x \log x \log \log x}{T},
\end{equation}
with $\delta(\chi)=1$ if $\chi$ is principal and $\delta(\chi)=0$ otherwise. With $R(C,\alpha_2, \alpha_1)$ in Table~\ref{tab:min}. 
\end{lemma}
\begin{table}[H]
    \begin{tabular}{ | l | l | l | l | }
    \hline
     $C$ & $\alpha_1$ & $\alpha_2$ & $R$ \\  
    \hline
    $3.4 $ & $ 1$ & $ 1$ & $ 35.4$ \\
    \hline
     $3.7 $ & $ 1$ & $ 2$ & $31.7 $ \\
     \hline
     $ 4.3$ & $ 1$ & $5 $ & $27.4 $ \\
    \hline
     $3.8$ & $2 $ & $1 $ & $20.3 $ \\
    \hline
     $4.3 $ & $ 2$ & $4 $ & $35.5 $ \\
    \hline
     
    \end{tabular}
    \quad
    \begin{tabular}{ | l | l | l | l | }
    \hline
     $C$ & $\alpha_1$ & $\alpha_2$ & $R$ \\  
    \hline
    $ 4$ & $3 $ & $ 1$ & $30.8 $ \\
    \hline
    $4.6 $ & $3 $ & $5 $ & $34.7 $ \\
    \hline
     $4.3 $ & $4 $ & $ 1$ & $22.5 $ \\
    \hline
     $4.9 $ & $4 $ & $ 6$ & $21.5 $ \\
    \hline
     $4.5 $ & $5 $ & $1 $ & $29.1 $ \\
    \hline
     
    \end{tabular}
    \quad
    \begin{tabular}{ | l | l | l | l | }
    \hline
     $C$ & $\alpha_1$ & $\alpha_2$ & $R$ \\  
    \hline
    $5.1 $ & $5 $ & $7 $ & $30.1 $ \\
    \hline
    $4.7 $ & $ 6$ & $ 1$ & $26.1 $ \\
    \hline
     $ 5.3 $ & $ 6$ & $8 $ & $23.7 $ \\
    \hline
     $4.9 $ & $7 $ & $1 $ & $20.2 $ \\
    \hline
     $5.5 $ & $7 $ & $9 $ & $21.1 $ \\
    \hline
     
    \end{tabular}
    \quad
\caption{Upper bound for $R(x,\alpha_2, \alpha_1)$}
\label{tab:min}   
\end{table}
Note that using Theorem \ref{the:func} we can compute $R$ for different values of $C$, $\alpha_1$ and $\alpha_2$. 
Lemma \ref{lemma:main} improves Lemma 2.1 in \cite{Yamada}, in the error term, reducing a $\log x$ factor to a $\log \log x$ one. It is interesting to note that Littlewood, in \cite{Littlewood}, assuming the Riemann Hypothesis proved the above result for $\psi(x)$ with an error term of $\frac{x\log x}{T}$. This suggests that even if it should be possible to improve the error term in the above result, it will probably be highly complicated.
Yamada's Lemma 2.1 is based on Lemma 1 in \cite{Chen-Wang}, and splitting a sum in a similar way as done by Dudek in \cite{Dudek} and being more careful with the error terms, it is possible to obtain an upper bound for $R(x,\alpha_1,\alpha_2)$ that is half in size compared to that of Yamada. We will not give more details on this as this result is superseded by the one obtained making Goldston's result in \cite{Goldston} explicit.
Here we will prove a fundamental explicit version of Lemma 2 in \cite{Goldston}, with a partially different proof to better control the error term.
\begin{lemma}
\label{lemma:ll}
For any $x\ge 3$ we have
\begin{equation*}
\sum_{\frac{x}{2}<n\le x-1.5}\frac{\Lambda(n)}{(xn^{-1}-1)}\le r_3(x) \quad \text{and} \quad\sum_{x+1.5\le n< 2x}\frac{\Lambda(n)}{(1-xn^{-1})}  \le 2r_3(x),
\end{equation*}
with
\begin{align*}
r_3(x):= x \log x& \Big(\frac{2}{3}\left(4+\frac{1}{\log 2}\right)+2\log \log x+\frac{2}{\log x}+ \frac{16}{15}\log \left(\frac{\log x}{\log 2}\right)\\ &+\left(\frac{\pi^2}{6}-1 \right)\left(\frac{2}{1+\sqrt{x-1.5}-\sqrt{x}}-\frac{1}{\sqrt{x}}\left(1-\log \left(\frac{\sqrt{2}-1}{\sqrt{2}+1}\right)\right)\right)\Big).
\end{align*}
\end{lemma}
\begin{proof}
The proof is based on Lemma 2 in \cite{Goldston}, with a partially different proof to obtain a better explicit bound. We will prove the first of the two inequalities, as the second follows basically in the same way.
We have 
\begin{equation}
\label{eq:split}
\sum_{\frac{x}{2}<n\le x-1.5}\frac{\Lambda(n)}{(xn^{-1}-1)}\le x  \sum_{1 \le l \le \frac{\log x}{\log 2}}\sum_{\frac{x}{2}<p^l\le x-1.5}\frac{\log p}{(x-p^l)}.
\end{equation}
We will first bound the right-hand side sum with $l=1$.
Now, letting $\pi(x)$ count the number of primes less than $x$ and define $P(x,y):=\pi(x)-\pi(x-y)$, in \cite{Montgomery} it is proved
\begin{equation}
\label{eq:P}
P(x,y) \le \frac{2y}{\log y} \quad 1< y \le x.
\end{equation}
We also have, with $2k$ the nearest even integer to $x$,
\begin{equation*}
\sum_{\frac{x}{2}<p\le x-1.5}\frac{1}{(x-p)}\le \sum_{k<\le 2k-1}\frac{1}{(x-p)}+2/3.
\end{equation*}
Now 
\begin{equation*}
\sum_{k\le p \le 2k-1}\frac{1}{(2k-p)}=\sum_{n=1}^k\frac{1}{n}\left(P(2k,n)-P(2k,n-1) \right)
\end{equation*}
\begin{equation*}
=\sum_{n=1}^{k-1}P(2k,n)\left(\frac{1}{n}-\frac{1}{n+1} \right)+\frac{1}{k}\left(P(2k,k)-P(2k,0)\right)
\end{equation*}
and using \eqref{eq:P},
\begin{equation*}
\le 2\sum_{n=2}^{k-1}\frac{1}{(n+1)\log n}+2+\frac{2}{\log k}   
\le \frac{2}{3\log 2}+2\log \log k +2 +\frac{2}{\log k}.
\end{equation*}
We now bound the right-hand side sum in \eqref{eq:split} for $l\ge 2$.
We have
\begin{equation}
\label{eq:right}
 \sum_{2 \le l \le \frac{\log x}{\log 2}}\sum_{\frac{x}{2}<p^l\le x-1.5}\frac{\log p}{(x-p^l)} \le \log x \sum_{2 \le l \le \frac{\log x}{\log 2}}\frac{1}{l}\sum_{\frac{x}{2}<p^l\le x-1.5}\frac{1}{(x-p^l)}.
\end{equation}
For $l \ge 2$, it is easy to see that
\begin{equation*}
\sum_{\frac{x}{2}<p^l\le x-1.5}\frac{1}{(x-p^l)}\le \int_{(\frac{x}{2})^{1/l}}^{ (x-1.5)^{1/l}-1 }\frac{1}{(x-t^l)}dt+\frac{2}{3}+\frac{2}{5} \le \frac{2}{l}\int_{\sqrt{\frac{x}{2}}}^{\sqrt{x-1.5}-1 }\frac{1}{(x-y^2)}dy+\frac{16}{15} ,
\end{equation*}
where in the last step we used the change of variables $t=y^{2/l}$. Noting that
\begin{equation*}
\int \frac{1}{(x-y^2)} dy=- \frac{\log \frac{|y-\sqrt{x}|}{|y+\sqrt{x}|}}{2\sqrt{x}},
\end{equation*}
we have 
\begin{align*}
 \sum_{2 \le l \le \frac{\log x}{\log 2}}\frac{1}{l}\sum_{\frac{x}{2}<p^l\le x-1.5}\frac{1}{(x-p^l)}\le \frac{16}{15}\log \left(\frac{\log x}{\log 2}\right)+\left(\frac{\pi^2}{6}-1 \right) \left(- \frac{\log \frac{|y-\sqrt{x}|}{|y+\sqrt{x}|}}{\sqrt{x}}\right)\bigg|_{\sqrt{\frac{x}{2}}}^{\sqrt{x-1.5}-1}.
\end{align*}
We can conclude the proof using that, for $ x>0$, $\log x \ge 1-1/x$.
\end{proof}
Note that it should be possible to improve the above result by a factor of $2/3$ bounding the two sums together.
We now introduce a variation of Chen and Wang Lemma 1 in \cite{Chen-Wang}, this is obtained using Lemma \ref{lemma:ll}. Being a bit more careful than Chen and Wan we also obtain a $\log 2$ saving, not counting the use of Lemma \ref{lemma:ll}.
\begin{lemma}
\label{lemma:CW1}
Take $f(s)=\sum_{n=1}^{\infty} \frac{\Lambda (n)}{n^s}$, that is absolutely converging for $\mathbb{R}(s)=\sigma>1$. Then for any $b>1$, $T\ge 1$ and $x=N+1/2 \ge 3$, with $N$ a positive integer, we have
\begin{equation*}
\sum_{n\le x} \Lambda(n) = \frac{1}{2\pi i}\int_{b-iT}^{b+iT}f(s)s^{-1}x^sds+R_{1}(x,T,b),
\end{equation*}
with
\begin{align*}
|R_{1}(x,T,b)|\le & \frac{x^b}{\pi T \log 2}\sum_{n=1}^{\infty} \frac{\Lambda(n)}{n^b}+\frac{2^{-1}x^b}{\pi T}\sum_{\frac{x}{2}<n<x-1/2}\Lambda(n)n^{-b}
 \\&+\frac{1}{\pi T}\Big(2^b\Big(\log(x-1/2) (x-1/2) \left(2+\frac{1}{x-1/2}\right)\left(\frac{x}{x-1/2}\right)^b
\\&+\log(x+1/2)  2(x+1/2)\left(\frac{x}{x+1/2}\right)^b+r_3(x)\Big)+2r_3(x)\Big).
\end{align*}
\end{lemma}
\begin{proof}
Following the proof of  Lemma 1 in \cite{Chen-Wang} we obtain
\begin{equation*}
|R_{1}(x,T,b)|\le \frac{x^b}{\pi T \log 2}\sum_{n=1}^{\infty} \frac{\Lambda(n)}{n^b}+\frac{1}{\pi T}\sum_{\frac{x}{2}<n<2x}\Lambda(n)\frac{(xn^{-1})^b}{\log(xn^{-1})}.
\end{equation*}
We are now left with obtaining an upper bound for the right-hand side sum. We start splitting the sum in two at $n=N$, obtaining
\begin{equation}
\label{eq:N}
\sum_{\frac{x}{a}<n<ax}\Lambda(n)\frac{(xn^{-1})^b}{\log(xn^{-1})}=\sum_{\frac{x}{2}<n<N}\Lambda(n)\frac{(xn^{-1})^b}{\log(xn^{-1})}+\sum_{N\le n<2x}\Lambda(n)\frac{(xn^{-1})^b}{\log(xn^{-1})}.
\end{equation}
The first term of \eqref{eq:N}, remembering that $x=N+1/2$, is
\begin{equation*}
\sum_{\frac{x}{a}<n<x-1/2}\Lambda(n)\frac{(xn^{-1})^b}{\log(xn^{-1})}.
\end{equation*}
We use that, for $x \ge 1$, by the Taylor expansion for $\log x$ it is easy to see that 
\begin{equation}
\label{eq:lg}
\log x\ge \frac{2(x-1)}{(x+1)},
\end{equation}
this gives
\begin{equation*}
\le 2^{-1}\sum_{\frac{x}{2}<n<x-1/2}\Lambda(n)\frac{(xn^{-1})^b(xn^{-1}+1)}{(xn^{-1}-1)} 
\end{equation*}
\begin{equation*}
=2^{-1}x^b\sum_{\frac{x}{2}<n<x-1/2}\Lambda(n)n^{-b}
+2^{b}\sum_{\frac{x}{a}<n<x-1/2}\frac{\Lambda(n)}{(xn^{-1}-1)}.
\end{equation*}
We can now bound the right-hand side sum using Lemma \ref{lemma:ll}.
The second term of \eqref{eq:N}, remembering that $x=N+1/2$, is
\begin{equation*}
\sum_{x-1/2\le n<2x}\Lambda(n)\frac{(xn^{-1})^b}{|\log(xn^{-1})|}\le \log(x-1/2) \frac{(x/(x-1/2))^b}{|\log(x/(x-1/2))|}
\end{equation*}
\begin{equation*}
+\log(x+1/2) \frac{(x/(x+1/2))^b}{|\log(x/(x+1/2))|}+\sum_{x+1.5\le n<2x}\frac{\Lambda(n)}{|\log(xn^{-1})|}.
\end{equation*}
We can bound the right-hand side sum using that, for $0\le x \le 1$, $|\log (1-x)| \ge x$ and Lemma \ref{lemma:ll}.
We thus obtain an upper bound for \eqref{eq:N} and conclude the proof.
\end{proof}
Thinking about possible improvements it is interesting to note that in the above result it would be possible to obtain a slightly better result splitting \eqref{eq:N} at a point different than $2$, but this would require proving a customized variation of Lemma \ref{lemma:ll}.
We can now prove Lemma \ref{lemma:main}, that proceeds similarly to Theorem 8 in \cite{Liu-Wang} and Lemma 2.1 in \cite{Yamada}.
\begin{theorem}
\label{the:func}
Let $\chi$ be a Dirichlet character modulo $q$, $c\ge 1$ and  $x \ge 1126$, then 
\begin{align}
\label{eq:imp}
|\psi(x,\chi)-\delta(\chi)x|\le  \sum_{\rho \in z(\chi),|\gamma|\le T} \frac{x^{\beta}}{|\rho|}+R^*(x,T,q),
\end{align}
with
\begin{align*}
R^*(x,T,q):=& \frac{\log q}{\log 2}+|R_2(T,x)|+|R_3(T,x)|+\log 2+R_5(T,x)+R_7(T,q,x)+\\&+R_8(T,q,x)
+\log x+\frac{x}{T-1}r_4(T,q)+R_{11}(T,q,x),
\end{align*}
\begin{align*}
r_4(T,q):=\frac{T+1}{\pi}\log \frac{q(T+1)}{2 \pi e}-\frac{T-1}{\pi}\log \frac{q(T-1)}{2 \pi e}+r_1(T+1,q)+r_1(T-1,q),
\end{align*}
\begin{align*}
|R_3(x,T)|\le & \frac{1}{\pi T \log 2}\Big(\log x - C + \frac{0.1877}{\log x} +2^{1+\frac{1}{\log x}}(x\log x+1.5x-1/2)\log x \\ &+x(\log x +\log 2 +2)\log 2x \Big),
\end{align*}
\begin{align*}
&R_5(T,x):=\frac{1}{2\pi}\left(\log x-C+\frac{0.1877}{\log x}\right)(T+1)\exp \left( \frac{\log(x+1/2)}{\log x}\right),
\end{align*}
\begin{align*}
r_5(x,\sigma,q):=&(2-\sigma)\cdot \\& \cdot \left(1.75\log (qx)+\frac{1}{2.5+x^2}+\frac{1}{x(2.25+x^2)^{1/2}}+\frac{\pi}{4x}+3.31 \right)+0.62,
\end{align*}
\begin{align*}
 R_7(T,q,x):= & \frac{(1.5+\frac{1}{\log x})(ex+2.5^{1+\frac{1}{\log x}})}{2\pi(T-1)}\cdot\\& \cdot \left( r_4(T+1,q)\left(r_4(T+1,q)+1\right)+\max_{T-1\le x \le T+1}r_5(x,-1/2,q)\right),
\end{align*}
\begin{align*}
r_6(x):=\left(1.75\log (2+|x|)+\frac{1}{2.5+x^2}+2.43+\frac{1}{((4+x^2)(0.25+x^2))^{1/2}}\right)+0.62,
\end{align*}
\begin{align*}
R_8(T,q,x):=& \frac{\left(\frac{1}{\sqrt{x}}+\frac{1}{\sqrt{2.5}} \right)(1.5+\frac{1}{\log x})}{2\pi(T-1)}\left( 2r_4(T+1,q)+\max_{-(T+1)\le x \le T+1}r_6(x)\right),
\end{align*}
\begin{align*}
R_{11}(T,q,x):=& (\sqrt{x}+2)\frac{\sqrt{q} \log^2q}{100}+(\sqrt{x}+1)\Big(R_0r_1(q,1)\log q+\frac{r_1(q,T+1)}{T+1}\\& +\frac{\log T}{\pi}\log \frac{q\sqrt{T}}{2\pi e}+r_1(q,T)\Big).
\end{align*}
\end{theorem}
\begin{proof}
Aside for the result in Lemma \ref{lemma:CW1} the proof is a more general version of Theorem 8 in \cite{Liu-Wang}.
If $\chi \pmod q$ is induced by the primitive character $\chi_1 \pmod {q_1}$ we have,
\begin{equation}
\label{eq:prim}
|\psi(\chi,x)-\psi(\chi_1,x)|\le \frac{\log q}{\log 2},
\end{equation}
and thus, taking note of the above error term, we can focus on primitive characters. Note that if $\chi=\chi_0$, then the following argument holds with $\psi(\chi,x)$ replaced by $\psi(\chi,x)-x$.
By Corollary 2.1 in \cite{Bennett1}, we have
\begin{equation}
\label{eq:TT0}
\sum_{|\gamma-T|\le 1}1\le  r_4(T,q),
\end{equation} 
thus it is easy to see that there exist a real $T_0$, such that $|T-T_0|\le 1$ and 
\begin{equation}
\label{eq:TT1}
\frac{1}{|\gamma-T_0|}\le r_4(T,q)+1 ,
\end{equation}
for any non-trivial zero $\rho=\beta +i\gamma$ of $L(s,\chi)$.
Defining $x_0=\lfloor x \rfloor+1/2$, we have
\begin{equation*}
\psi(x,\chi)=\psi(x_0,\chi).
\end{equation*}
We have
\begin{equation}
\psi(x,\chi)=\frac{1}{2\pi i}\int_{b-iT_0}^{b+iT_0}\left(-\frac{L'}{L}(s,\chi) \right)\frac{x_0^s}{s}ds+R_2(x,T),
\end{equation}
with $R_2(x,T)$ with an explicit upper bound given by Lemma \ref{lemma:CW1}, with $b=1+\log^{-1}x$, and \eqref{eq:z'/z}.
By Lemma 1 in \cite{Chen-Wang} and always using \eqref{eq:z'/z}, we obtain
\begin{equation}
\psi(2.5,\chi)=\frac{1}{2\pi i}\int_{b-iT_0}^{b+iT_0}\left(-\frac{L'}{L}(s,\chi) \right)\frac{2.5^s}{s}ds+R_3(x,T).
\end{equation}
Considering the difference between $\psi(x,\chi)$ and $\psi(2.5,\chi)$, we obtain
\begin{equation}
\label{eq:diff}
\psi(x,\chi)=\frac{1}{2\pi i}\int_{b-iT_0}^{b+iT_0}\left(-\frac{L'}{L}(s,\chi) \right)\frac{x_0^s-2.5^s}{s}ds+R_4(x,T),
\end{equation}
with, observing that $|\psi(2.5,\chi)|\le \log 2$, $|R_4(x,T)|\le |R_2|+|R_3|+\log 2$. 
The difference between the main term of the right-hand side of \eqref{eq:diff} and its analogue with $x_0$ replaced by $x$ is
\begin{equation*}
\le \frac{1}{2\pi}\int_{-T_0}^{T_0}\left|\frac{L'}{L}(b+iu,\chi) \right| \left| \int_{x_0}^{2.5}x^{b-1+iu}dx\right|du \le R_5(T,x),
\end{equation*}
in the last step we used \eqref{eq:z'/z}, $b=1+\log^{-1} x$ and $|T_0-T|\le 1$.
Thus \eqref{eq:diff} becomes
\begin{equation}
\label{eq:diff1}
\psi(x,\chi)=\frac{1}{2\pi i}\int_{b-iT_0}^{b+iT_0}\left(-\frac{L'}{L}(s,\chi) \right)\frac{x^s-2.5^s}{s}ds+R_6(x,T),
\end{equation}
with $|R_6(x,T)|\le |R_4|+R_5$.
We now take $\Omega$ to be a rectangle with vertices $b\pm iT_0$ and $1/2 \pm iT_0$. By Cauchy's residue theorem we obtain
\begin{equation}
\label{eq:Cauchy}
\frac{1}{2\pi i} \int_{\Omega} \left( -\frac{L'}{L}(s,\chi)\right)\frac{x^s-2.5^s}{s}ds=-\sum_{|\gamma|\le T_0} \frac{x^{\varrho}-2.5^{\varrho}}{\varrho}+\theta\log x,
\end{equation}
with $|\theta| \le 1$.
By  Lemma 9 in \cite{Chen-Wang}, \eqref{eq:TT0} and \eqref{eq:TT1}, for $-1/2 \le \sigma \le b$, we obtain
\begin{equation*}
\left|\frac{L'}{L}(\sigma +iT_0,\chi)\right| \le r_4(T+1,q)\left(r_4(T+1,q)+1\right)+r_5(T_0,\sigma,T,q).
\end{equation*}
Also by (9') in \cite{Chen-Wang} and \eqref{eq:TT0}
\begin{equation*}
\left|\frac{L'}{L}(-1/2 +ix,\chi)\right| \le 2r_4(T,q)+r_6(x, T).
\end{equation*}
Hence 
\begin{equation*}
\left|\frac{1}{2\pi i} \int_{-1/2 \pm iT_0}^{b\pm iT_0} \left( -\frac{L'}{L}(s,\chi)\right)\frac{x^s-2.5^s}{s}ds\right| \le R_7(T,q,x)
\end{equation*}
and
\begin{equation*}
\left|\frac{1}{2\pi i} \int_{-1/2 - iT_0}^{-1/2 + iT_0} \left( -\frac{L'}{L}(s,\chi)\right)\frac{x^s-2.5^s}{s}ds \right|\le R_8(T,q,x).
\end{equation*}
Thus by \eqref{eq:diff1} and \eqref{eq:Cauchy} 
\begin{equation}
\label{eq:sum}
\psi(t,\chi)=-\sum_{|\gamma|\le T_0} \frac{x^{\varrho}-2.5^{\varrho}}{\varrho}+R_9(T,q,x)
\end{equation}
with $|R_9(T,q,x)|\le |R_6|+R_7+R_8+\log x$.
By \eqref{eq:TT0}, the difference between $\sum_{|\gamma|\le T_0} \frac{x^{\varrho}}{\varrho}$ and $\sum_{|\gamma|\le T} \frac{x^{\varrho}}{\varrho}$
 is $\le \frac{x}{T-1}r_4(T,q)$. Thus \eqref{eq:sum} can be rewritten as
 \begin{equation}
\label{eq:sum1}
\psi(t,\chi)=-\sum_{|\gamma|\le T} \frac{x^{\varrho}}{\varrho}+\sum_{|\gamma|\le T_0} \frac{2.5^{\varrho}}{\varrho}+R_{10}(T,q,x),
\end{equation}
 with $|R_{10}(T,q,x)|\le |R_9(T,q,x)|+\frac{x}{T-1}r_4(T,q)$.
It is easy to see that 
  \begin{equation}
\label{eq:sum2}
\left|\sum_{|\gamma|\le T_0} \frac{2.5^{\varrho}}{\varrho}\right|\le 2.5\sum_{|\gamma|\le T+1} \frac{1}{|\varrho|}\le 2.5 \left(\sum_{|\gamma|\le 1} \frac{1}{\mathbf{Re} (\varrho)}+\sum_{1<|\gamma|\le T+1} \frac{1}{|\varrho|} \right).
\end{equation}
Now using  Lemma \ref{lemma:B} to bound the possible two Siegel zeros, Theorem \ref{theo:kadiri} to bound the other zeroes and Lemma \ref{lemma:N}, we obtain
\begin{equation}
\label{eq:<1}
\sum_{|\gamma|\le 1} \frac{1}{\mathbf{Re} (\varrho)}\le \frac{\sqrt{q} \log^2q}{50}+R_0r_1(q,1)\log q .
\end{equation}
Furthermore, by Lemma \ref{lemma:N}
\begin{equation}
\label{eq:>1}
\sum_{1<|\gamma|\le T+1} \frac{1}{|\varrho|}\le \frac{r_1(q,T+1)}{T+1}+\int_1^{T+1} \frac{r_1(q,y)}{y^2}dy.
\end{equation}
We can also see that
\begin{equation}
\label{eq:sum3}
\sum_{|\gamma|\le T, \beta <1/2} \frac{x^{\varrho}}{\varrho}=\sum_{|\gamma|\le 1, \beta <1/2} \frac{x^{\varrho}}{\varrho}+\sum_{1<|\gamma|\le T+1, \beta <1/2} \frac{x^{\varrho}}{\varrho}.
\end{equation}
Similarly to \eqref{eq:<1}, as there is only one possible Siegel zero in this range, we have
\begin{equation}
\label{eq:<11}
\sum_{|\gamma|\le 1, \beta <1/2} \frac{x^{\sigma}}{|\varrho|}\le \sqrt{x}\left(\frac{\sqrt{q} \log^2q}{100} +R_0r_1(q,1)\log q\right)
\end{equation}
and similarly to \eqref{eq:>1}
\begin{equation}
\label{eq:>11}
\sum_{1<|\gamma|\le T+1, \beta <1/2} \frac{x^{\sigma}}{|\varrho|} \le \sqrt{x}\left( \frac{r_1(q,T+1)}{T+1}+\int_1^{T+1} \frac{r_1(q,y)}{y^2}dy\right).
\end{equation}
Here we can note that
\begin{equation*}
\int_1^{T+1} \frac{r_1(q,y)}{y^2}dy\le \frac{\log T}{\pi}\log \frac{q\sqrt{T}}{2\pi e}+r_1(T,q).
\end{equation*}
Thus by \eqref{eq:sum1}, \eqref{eq:sum2} together with \eqref{eq:<1} and \eqref{eq:>1} and \eqref{eq:sum2} together with \eqref{eq:<11} and \eqref{eq:>11}, we obtain
\begin{equation}
\label{eq:fin}
\left|\psi(t,\chi)+\sum_{|\gamma|\le T, \beta \ge 1/2} \frac{x^{\varrho}}{\varrho}\right| \le |R_{10}| + R_{11}(T,q,x).
\end{equation}
This concludes the proof.
\end{proof}
We can now easily prove Lemma \ref{lemma:main}.
\begin{proof}(Lemma \ref{lemma:main})
The result follows easily by Theorem \ref{the:func}, the choices for $q, T, x$ done in Lemma \ref{lemma:main} and simple computations. We also used that
\begin{align*}
\max_{T-1\le x \le T+1}&r_5(x,-1/2,q) \le 2.5\Big(1.75\log (q(T+1))+\frac{1}{2.5+(T-1)^2}\\&+\frac{1}{(T-1)(2.25+(T-1)^2)^{1/2}}+\frac{\pi}{4\cdot (T-1)}+3.31 \Big)+0.62,
\end{align*}
\begin{align*}
&\max_{-(T+1)\le x \le T+1}r_6(x)\le 2.5\left(1.75\log (2T+3)+\frac{1}{2.5}+2.43+\frac{1}{(4\cdot 0.25)^{1/2}}\right)+0.62.
\end{align*}
\end{proof}
\section{Proof of Theorem \ref{theo:1}}
\label{sec:theo1}
Note that by Theorem 3.6 in \cite{McCurley1}, \eqref{eq:1} and \eqref{eq:2} are equivalent and we will thus focus on proving \eqref{eq:1}.
We set
\begin{equation*}
\Sigma=\sum_{\chi} \sum_{\rho \in z(\chi), |\lambda|\le T}\frac{x^{\beta-1}}{|\rho|},
\end{equation*}
Now we need to bound $\Sigma$ and to do this successfully we split the sum as follows.
For $H> 1$ and $R>0$, we define
\begin{equation*}
z_0(\chi, H,R)=\sharp \lbrace \rho \mid \frac{1}{2}\le \beta \le 1-\frac{1}{R\log qH}, |\gamma|<H, L(\rho,\chi)=0 \rbrace
\end{equation*}
and
\begin{equation*}
z_1(\chi, H,R)=\sharp \lbrace \rho \mid \frac{1}{R_0} \le (1-\beta)\log qH \le \frac{1}{R}, |\gamma|<H, L(\rho,\chi)=0 \rbrace.
\end{equation*}
We define
\begin{equation*}
\Sigma_0=\sum_{\chi} \sum_{\rho \in z_0(\chi, T, R)}\frac{x^{\beta-1}}{|\rho|} \quad \text{and} \quad
\Sigma_1=\sum_{\chi} \sum_{\rho \in z_1(\chi, T, R)}\frac{x^{\beta-1}}{|\rho|}.
\end{equation*}
This gives us 
\begin{equation*}
\Sigma= \Sigma_0+\Sigma_1+E_0\frac{x^{\beta_0-1}}{\beta_0}.
\end{equation*}
It is easy to see that
\begin{equation*}
\sum_{\rho \in z_0(\chi, T, R)}\frac{x^{\beta-1}}{|\rho|} \le \frac{1}{2}x^{-1/R\log qT}\left( 2N(\chi,1)+\int_1^T \frac{dN(\chi,t)}{t}\right)
\end{equation*}
and thus, by Lemma \ref{lemma:N},
\begin{equation}
\label{eq:s0}
\Sigma_0\le \frac{qS(T,q)}{2}x^{-1/R\log qT},
\end{equation}
with
\begin{align*}
S(T,q)=\frac{1}{\pi}\left((2+\log T)\log \frac{q}{2\pi e}+T(\log T-1)+1 \right)+2r_1(q,1)+\log T r_1(q,\sqrt{T}).
\end{align*}
We are now left with estimating $\Sigma_1$. We start with the following easy estimate, that is Lemma 2.2 in \cite{Yamada} with the corrected upper bound. In Lemma 2.2 the condition $\exp \sqrt{\frac{\log x}{R_0}} \le (qT)^{1/R_0\lambda}$ is misstated as $\exp \sqrt{\frac{x}{R_0}} \le (qT)^{1/R_0\lambda}$.
\begin{lemma}
\label{lemma:mu}
Let $\rho=\beta+i \gamma$ de a zero of $L(s,\chi)$, with $\beta <1-1/R_0\log q|\gamma|$, $|\gamma|\le T$ and $\rho\neq 0$. Let $\lambda$ be such that $\beta=1-\lambda/\log qT$, then we have
\begin{equation}
\label{eq:3ways}
 \left| \frac{x^{\rho-1}}{\rho}\right| \le \mu(\lambda):= \begin{cases}
    \frac{x^{-1/R_0\log q}}{1-1/R_0\log q} & \text{if}~ |\gamma|\le 1,\\
     q \exp-2\sqrt{\frac{\log x}{R_0}} & \text{if}~ |\gamma|>1 ~\text{and}~ \exp \sqrt{\frac{\log x}{R_0}} < (qT)^{1/R_0\lambda},\\
     q\frac{x^{-\lambda/\log qT}}{(qT)^{1/R_0\lambda}} & \text{otherwise}.
  \end{cases}
\end{equation}
\end{lemma}
\begin{proof}
When $|\gamma| \le 1$ we obtain
\begin{equation*}
\left| \frac{x^{\rho-1}}{\rho}\right| \le  \frac{x^{\beta-1}}{\beta}\le \frac{x^{-1/R_0\log q}}{1-1/R_0\log q} ,
\end{equation*}
using in the last step that $ \beta \ge \frac{1}{R_0 \log q}\ge \frac{1}{\log x}$, which follows from Theorem \ref{theo:kadiri}.\newline
We may now assume $\gamma \ge 0$. If $\gamma\ge 1$, observing that
\begin{equation*}
\frac{1}{R_0 \log q\gamma}+\frac{\log \gamma}{\log x}\ge 2 \sqrt{\frac{1}{R_0 \log x}}-\frac{\log q}{\log x},
\end{equation*}
we obtain
\begin{equation*}
\left| \frac{x^{\rho-1}}{\rho}\right| \le  x^{\beta-1}\le q \exp-2\sqrt{\frac{\log x}{R_0}}.
\end{equation*}
Let $\gamma_0=\frac{1}{q}\exp \sqrt{\frac{\log x}{R_0}}$ and $\gamma_1=(q)^{1/R_0\lambda-1}(T)^{1/R_0\lambda}$. We may assume $t_0\ge t_1$. In the case $\gamma \le \gamma_1$ we have
\begin{equation*}
\frac{x^{1/R_0\log q \gamma}}{\gamma}\le \frac{x^{1/R_0\log q \gamma_1}}{\gamma_1}=q\frac{x^{-\lambda/\log qT}}{(qT)^{1/R_0\lambda}},
\end{equation*}
since $\frac{x^{1/R_0\log q \gamma}}{\gamma}$ is increasing below $\gamma_0$. In the other case we easily have
\begin{equation*}
\frac{x^{\lambda/\log qT \gamma}}{\gamma}\le \frac{x^{\lambda/\log qT \gamma}}{\gamma_1}=q\frac{x^{-\lambda/\log qT}}{(qT)^{1/R_0\lambda}}.
\end{equation*}
\end{proof}
We now split $\Sigma_1$ as follows
\begin{equation}
\label{eq:s1split}
\Sigma_1=\sum_{\chi} \sum_{\rho \in z_1(\chi, T, R), |\gamma| \le 1}\frac{x^{\beta-1}}{|\rho|}+\sum_{\chi} \sum_{\rho \in z_1(\chi, T, R), |\gamma| \ge 1}\frac{x^{\beta-1}}{|\rho|}.
\end{equation}
By Lemma \ref{lemma:N}, we obtain
\begin{equation}
\label{eq:s11}
\sum_{\chi} \sum_{\rho \in z_1(\chi, T, R), |\gamma| \le 1}\frac{x^{\beta-1}}{|\rho|}\le \left(\frac{1}{\pi}\log \frac{q}{2\pi e}+ r_1(q,1)\right) \frac{x^{-1/R_0\log q}}{1-1/R_0\log q}.
\end{equation}
Now let $p_i=\beta_i+i\gamma_i$ be all zeroes of $\prod (s,q)$ with $\beta_i=1-\lambda_i/\log qT$, $1\le \gamma_i \le T$ and $ 1/R_{0}<\lambda_1 \le \lambda_2 \cdots$ \newline
We now summarize some of the results in \cite{Liu-Wang} regarding zeroes of $\prod (s,q)$.
\begin{lemma}
\label{lemma:LW}
Assuming $qT \ge 8\cdot 10^{9}$, then $\lambda_1 \le \eta_i$ and $\lambda_2 > \xi_i$ for each $i=1,\cdots,12$ and $\eta_i$, $\xi_i$ from Table \ref{LW1} and $\lambda_3 \ge 0.26213$. Moreover, if $qT\ge 10^{11}$ then $\lambda_n \ge \varrho_n$, with $\varrho_n$ in Table \ref{LW2}.
\end{lemma}
\begin{proof}
See Theorem 1-2 and Table 1, 3-5 in \cite{Liu-Wang}.
\end{proof}
\begin{table}[H]
\begin{center}
\begin{tabular}{ |c|c|c|}
\hline
 $i$ & $\eta_i$ & $\xi_i$ \\ 
 \hline
 \hline
   1  & 0.16 & 0.2605 \\
   2 & 0.17 & 0.2477 \\
   3 & 0.18 & 0.2356 \\
   4 & 0.19 & 0.2242 \\
   5 & 0.20 & 0.2135 \\
   6 & 0.206 & 0.2074 \\
   7 & 0.2067 & 0.2067 \\ 
   \hline              
\end{tabular}
\end{center}
\caption{Values $\eta_i$ and $\xi_i$. \label{LW1}}
\end{table}
\begin{table}[H]
\begin{center}
\begin{tabular}{ |c||c|c|c|c|c|c|c|c|c|c|c|c|}
\hline
 $n$ & 4 & 5&6&7&10&18& 45 &91&146&332&834&7000\\ 
 \hline
  $\varrho_n$ &0.28  &0.31 &0.32& 0.33& 0.36& 0.39&0.42&0.45&0.46&0.47&0.475&0.478 \\  
\hline              
\end{tabular}
\end{center}
\caption{Size $\lambda_n$. \label{LW2}}
\end{table}
 Now using \eqref{eq:3ways} and Lemma \ref{lemma:LW}, we can easily obtain an upper bound for $\Sigma_1$, with the sum restricted to $|\gamma|\ge 1$. We are thus summing $x^{\beta-1}/|\rho|$ over the zeroes $\rho$ of $L(s,\chi)$, with $1-1/R\log qT \le \beta_i \le 1-1/R_0\log qT$ and $1 \le |\gamma_i|\le T$.
 \begin{corollary}
 Assume $qT \ge 10^{11}$, then, for some $i=1,\cdots,6$, we have
 \begin{equation}
 \label{eq:E1}
 \sum_{\chi} \sum_{\rho \in z_1(\chi, T, R), |\gamma| \ge 1}\frac{x^{\beta-1}}{|\rho|}
 \end{equation}
 \begin{equation*} \le 2 q\min_{0 \le J \le 12} \left( \mu(\eta_i)+\mu(\xi_{i+1})+\sum_{j=1}^JM_j\mu(\max\{\xi_{i+1},\nu_j \})\right),
 \end{equation*}
 with $R=R_{i,J}=1/\max\{ \xi_{i+1},\nu_{J+1}\}$ and $M_j$ and $\nu_j$ defined in Table \ref{LW3}.
 \end{corollary}
 \begin{table}[H]
 \begin{center}
\setlength\tabcolsep{3pt}
\begin{tabular}{ |c||c|c|c|c|c|c|c|c|c|c|c|c|c|}
\hline
 $j$ & 1 & 2&3&4&5&6& 7 &8&9&10&11&12&13\\ 
\hline
 $\nu_j$ & 0.26213 & 0.27&0.30&0.32&0.33&0.36& 0.39 &0.42&0.45&0.46&0.47&0.475&0.478\\ 
 \hline
  $M_j$ &1  &1 &1& 1& 4& 7&47&57&55&186&502&6166&- \\  
\hline              
\end{tabular}
\end{center}
\caption{Sizes $\nu_j$ and $M_j$. \label{LW3}}
\end{table}
By Theorem \ref{theo:kadiri}, Theorem \ref{the:func}, \eqref{eq:s0} and \eqref{eq:E1} we can now prove Theorem \ref{theo:1}, with
\begin{equation}
\label{eq:C}
C(\alpha_1, \alpha_2, Y_0) = \max_{x\ge Y_0} \max_{q\le \log^{\alpha_1}x} \Big(\frac{R^*(x,T,q)T}{x\log^2 x}+ \left(\frac{1}{\pi}\log \frac{q}{2\pi e}+ r_1(q,1)\right) \cdot 
\end{equation}
\begin{equation*}
\cdot \frac{x^{-1/R_0\log q}}{1-1/R_0\log q}+q \log^{\alpha_2} x\max_{1\le i \le 6}\min_{0 \le J \le 12} \Big( \frac{1}{2}  S(T,q)\cdot
 \end{equation*}
\begin{equation*}
\cdot x^{-1/R_{i,J}\log qT}+2   \Big( \mu(\eta_i)+\mu(\xi_{i+1})+\sum_{j=1}^JM_j\mu(\max\{\xi_{i+1},\nu_j \})\Big)\Big).
\end{equation*}
Note that $qT \ge 10^{11}$, with our choice of $T$, holds for $ \log \log x \ge \frac{11 \log 10}{\alpha_1+\alpha_2+3}$.
\section{Proof of Theorem \ref{theo:2}}
\label{sec:theo2}
Let $\chi^*$ be the primitive character modulo $q*$, that induces $\chi$ modulo $q$, then it is easy to see that
\begin{equation*}
|\psi(x,\chi)-\psi(x,\chi*)|\le \log^2 qx,
\end{equation*}
 and, noting that $\psi(x, \chi_{0}^*)=\psi(x)$, 
 \begin{equation*}
 \left| \psi(x,q,a)-\frac{\psi(x)}{\varphi(q)}\right|\le \frac{1}{\varphi(q)} \left| \sum_{\chi \pmod q, \chi \neq \chi_0}\psi(x,\chi*)\right|+\log^2 qx.
 \end{equation*}
 Hence we easily obtain
 \begin{equation*}
 \sum_{q\le Q, q_0 \nmid q}\left| \psi(x,q,a)-\frac{\psi(x)}{\varphi(q)}\right|
 \end{equation*}
  \begin{equation*}
  \le Q\log^2Qx+\sum_{q\le Q, q_0 \nmid q}\frac{1}{\varphi(q)} \left| \sum_{\chi \pmod q, \chi \neq \chi_0}\psi(x,\chi^*)\right|
 \end{equation*}
  \begin{equation*}
\le Q\log^2Qx+\left(\sum_{1\le m\le Q}\frac{1}{\varphi(m)}\right) \sum_{1<q\le Q, q_0 \nmid q}\frac{1}{\varphi(q)}\left| \sum_{\chi \pmod q}^*\psi(x,\chi)\right|,
 \end{equation*}
 where $\sum_{\chi \pmod q}^*$ denotes the sum over all primitive characters $\chi \pmod q$. By Theorem A.17 in \cite{Nathanson}, we further obtain
 \begin{equation}
 \label{eq:sum}
\le Q\log^2Qx+\frac{c_1}{2}\log x \sum_{1<q\le Q, q_0 \nmid q}\frac{1}{\varphi(q)}\left| \sum_{\chi \pmod q}^*\psi(x,\chi)\right|.
 \end{equation}
 We now split the sum in \eqref{eq:sum} at $Q_1=\log^A x$. We start bounding the sum up to $Q_1$, by Theorem A.17 in \cite{Nathanson}, with
 \begin{equation*}
 \sum_{1<q\le Q_1, q_0 \nmid q}\frac{1}{\varphi(q)}\left| \sum_{\chi \pmod q}^*\psi(x,\chi)\right|
 \end{equation*}
  \begin{equation*}
  \le c_1(1+A\log \log x)\max_{1<q\le Q_1, q_0 \nmid q}\left| \sum_{\chi \pmod q}^*\psi(x,\chi)\right|.
 \end{equation*}
 We can now see that
 \begin{equation*}
 \sum_{\chi \pmod q}^*\psi(x,\chi)\le x\left| -1+\frac{1}{x}\sum_{\chi \pmod q}^*\psi(x,\chi) \right|+x\left|\frac{\psi(x,\chi_0)}{x}-1 \right|
 \end{equation*}
  \begin{equation*}
\le x\left| -1+\frac{1}{x}\sum_{\chi \pmod q}^*\psi(x,\chi) \right|+x\left|\frac{\psi(x)}{x}-1 \right|+\frac{\log q}{\log 2}
 \end{equation*}
 and, by Theorem \ref{theo:1}, \cite{Trudgian2} and $ \log \log x \ge \max \left\{7, \frac{11 \log 10}{2A} \right\}$, we have
 \begin{equation}
 \label{eq:Q1}
 \le \frac{C( A,A-3,X_0)x}{\log^{A-3}x}+\frac{x^{\beta_0}}{\beta_0}+34x(\log x)^{1.52}\exp \left(-0.81\sqrt{\log x} \right) +\frac{\log q}{\log 2}.
 \end{equation}
 Since $q_0\nmid q$ implies $q\le Q_1$ is not exceptional, by Theorem \ref{theo:kadiri}, we have
 \begin{equation}
 \label{eq:Q11}
\frac{x^{\beta_0}}{\beta_0}\le \frac{x^{1-\frac{1}{2AR_1 \log \log x}}}{1-\frac{1}{2AR_1 \log \log x}}
 \end{equation}
 We now want to bound the part of the sum in \eqref{eq:sum} from $Q_1$ to $Q$. We can easily do this as, by Theorem 1.2 in \cite{Amir} and partial summation formula, we obtain
 \begin{equation}
 \label{eq:Q2}
 \sum_{Q_1\le q\le Q}\frac{1}{\varphi(q)}\left| \sum_{\chi \pmod q}^*\psi(x,\chi)\right|\le \frac{4c_0 x\log ^{\frac{7}{2}}x}{Q_1}+
\end{equation}
\begin{equation*}
+c_0(\log x)^{\frac{7}{2}}\Big(4\sqrt{x}\log^A x+4\frac{x}{\log^A x} +18\frac{x^{\frac{11}{12}}}{\log ^{\frac{A}{2}}x}+5x^{\frac{5}{6}}+\frac{5}{2}x^{\frac{5}{6}}\log x\Big).
\end{equation*}
Now Theorem \ref{theo:2} follows from \eqref{eq:sum}, \eqref{eq:Q1}, \eqref{eq:Q11} and \eqref{eq:Q2}.

\section*{Acknowledgements}
I would like to thank my supervisor Tim Trudgian for his help in developing this paper and his insightful comments.

\begin{table}[H] 
\small
\caption{Upper bound for $C(\alpha_1,\alpha_2,Y_0)$}  
\label{tab:cn} 
    \begin{tabular}{ | l | l | l | l | }
    \hline
     $Y_0$ & $\alpha_1$ & $\alpha_2$ & $C $ \\  
    \hline
     $5.6 $ & $1$& $1$ & $11.34 $ \\
     $5.7 $ & $1$ & $1$ & $0.89 $ \\
     $7.4$ & $1$ & $1$ & $1.6 \cdot 10^{-5}$\\
     \hline
     $6.2 $ & $1$ & $2$ & $37.39 $ \\
     $6.3 $ & $1$ & $2$ & $0.83 $ \\
     $7.5 $ & $1$ & $2$ & $1.3 \cdot 10^{-5} $ \\
     \hline
     $ 6.7$ & $1$ & $3$ & $40.2 $ \\
     $6.8 $ & $1$ & $3$ & $0.21 $ \\
     $7.6 $ & $1$ & $3$ & $1.1 \cdot 10^{-5} $ \\
     \hline    
     $7.1 $ & $1$ & $4$ & $80.84 $ \\
     \hline
    \end{tabular}
    \quad
    \begin{tabular}{ | l | l | l | l |  }
    \hline
     $Y_0$ & $\alpha_1$ & $\alpha_2$ & $C  $ \\  
    \hline
     $7.2 $ & $1$& $4$ & $9.3 \cdot 10^{-2} $ \\
     $7.6 $ & $1$ & $4$ & $1.1 \cdot 10^{-5}$ \\
     \hline
     $7.5 $ & $1$ & $5$ & $1.42 $ \\
     $7.6 $ & $1$ & $5$ & $2.2\cdot 10^{-4} $ \\
     $7.7 $ & $1$ & $5$ & $9\cdot 10^{-6} $ \\
     \hline
     $7.8 $ & $1$ & $6$ & $1 $ \\
     $7.9 $ & $1$ & $6$ & $3.1 \cdot 10^{-5} $ \\
     $8 $ & $1$ & $6$ & $5.1 \cdot 10^{-6} $ \\
     \hline
     $8.1 $ & $1$ & $7$ & $1.2\cdot 10^{-2} $ \\
     $8.2 $ & $1$ & $7$ & 3.5$\cdot 10^{-6} $ \\
     \hline
    \end{tabular}
    \quad
    \begin{tabular}{ | l | l | l | l |  }
    \hline
     $Y_0$ & $\alpha_1$ & $\alpha_2$ & $C  $ \\  
    \hline
     $8.3 $ & $1$ & $7$ & $2.9 \cdot 10^{-6} $ \\
     \hline
     $8.3 $ & $1$ & $8$ & $0.92 $ \\
     $8.4 $ & $1$ & $8$ & $3\cdot 10^{-6} $ \\
     $8.5 $ & $1$ & $8$ & $2 \cdot 10^{-6} $ \\
     \hline
     $8.5 $ & $1$ & $9$ & $8.85 $ \\
     $8.6 $ & $1$ & $9$ & $2.9\cdot 10^{-6} $ \\
     $8.7 $ & $1$ & $9$ & $1.4\cdot 10^{-6} $ \\
     \hline
     $8.7 $ & $1$ & $10$ & $7.1 $ \\
     $8.8 $ & $1$ & $10$ & $1.3\cdot 10^{-6} $ \\
     $8.9 $ & $1$ & $10$ & $8.9\cdot 10^{-7} $ \\
     \hline
    \end{tabular}
    \quad
      
\end{table}
\begin{table}[H]
    \small
    \begin{tabular}{ | l | l | l | l |  }
    \hline
     $Y_0$ & $\alpha_1$ & $\alpha_2$ & $C  $ \\  
    \hline
     $6.4 $ & $2$ & $1$ & $34.77 $ \\
     $6.5 $ & $2$ & $1$ & $0.66 $ \\
     $6.9 $ & $2$ & $1$ & $4.3\cdot 10^{-5} $ \\
     \hline
     $6.9 $ & $2$ & $2$ & $7.54 $ \\
     $7 $ & $2$ & $2$ & $2.7\cdot 10^{-2} $ \\
     $7.3 $ & $2$ & $2$ & $2\cdot 10^{-5} $ \\
     \hline
     $7.3 $ & $2$ & $3$ & $2.21 $ \\
     $7.4 $ & $2$ & $3$ & $1.4\cdot 10^{-3} $ \\
     $7.6 $ & $2$ & $3$ & $1.1\cdot 10^{-5} $ \\
     \hline
     $7.6 $ & $2$ & $4$ & $17.1 $ \\
     $7.7 $ & $2$ & $4$ & $2.6 \cdot 10^{-3} $ \\
     $7.8 $ & $2$ & $4$ & $7.7 \cdot 10^{-6} $ \\
     \hline
     $7.9 $ & $2$ & $5$ & $5.32 $ \\
     $8 $ & $2$ & $5$ & $1.4 \cdot 10^{-4} $ \\
     $8.1 $ & $2$ & $5$ & $4.2 \cdot 10^{-6} $ \\
     \hline
     $ 8.2$ & $2$ & $6$ & $2.3\cdot 10^{-2} $ \\
     $8.3 $ & $2$ & $6$ & $2.9\cdot 10^{-6} $ \\
     $8.4 $ & $2$ & $6$ & $2.4\cdot 10^{-6} $ \\
     \hline
     $8.4 $ & $2$ & $7$ & $0.64 $ \\
     $8.5 $ & $2$ & $7$ & $2.3\cdot 10^{-6} $ \\
     $8.6 $ & $2$ & $7$ & $1.6\cdot 10^{-6} $ \\
     \hline
     $8.6 $ & $2$ & $8$ & $1.95 $ \\
     $8.7 $ & $2$ & $8$ & $1.5\cdot 10^{-6} $ \\
     $8.8 $ & $2$ & $8$ & $1.1\cdot 10^{-6} $ \\
     \hline
     $8.8 $ & $2$ & $9$ & $0.44 $ \\
     $8.9 $ & $2$ & $9$ & $9\cdot 10^{-7} $ \\
     $9 $ & $2$ & $9$ & $7.4\cdot 10^{-7} $ \\
     \hline
     $ 9$ & $2$ & $10$ & $4.1\cdot 10^{-3} $ \\
     $ 9.1$ & $2$ & $10$ & $6.1\cdot 10^{-7} $ \\
     $ 9.2$ & $2$ & $10$ & $5\cdot 10^{-7} $ \\
     \hline
    \end{tabular}
    \quad
    \begin{tabular}{ | l | l | l | l |  }
    \hline
     $Y_0$ & $\alpha_1$ & $\alpha_2$ & $C  $ \\  
    \hline
     $7.1 $ & $3$ & $1$ & $0.4 $ \\
     $7.2 $ & $3$ & $1$ & $8.5\cdot 10^{-4} $ \\
     $7.3 $ & $3$ & $1$ & $2.1\cdot 10^{-5} $ \\
     \hline
     $7.4 $ & $3$ & $2$ & $10.4 $ \\
     $7.5 $ & $3$ & $2$ & $1.3\cdot 10^{-2} $ \\
     $7.6 $ & $3$ & $2$ & $1.5\cdot 10^{-5} $ \\
     \hline
     $7.7 $ & $3$ & $3$ & $78.39 $ \\
     $7.8 $ & $3$ & $3$ & $1.2\cdot 10^{-2} $ \\
     $7.9 $ & $3$ & $3$ & $6.85\cdot 10^{-6} $ \\
     \hline
     $8 $ & $3$ & $4$ & $10.83 $ \\
     $8.1 $ & $3$ & $4$ & $2.4\cdot 10^{-4} $ \\
     $8.2 $ & $3$ & $4$ & $3.5\cdot 10^{-6} $ \\
     \hline
     $ 8.3$ & $3$ & $5$ & $1.7\cdot 10^{-2} $ \\
     $8.4 $ & $3$ & $5$ & $2.39\cdot 10^{-6} $ \\
     $8.5 $ & $3$ & $5$ & $2\cdot 10^{-6} $ \\
     \hline
     $8.5 $ & $3$ & $6$ & $0.17 $ \\
     $8.6 $ & $3$ & $6$ & $1.7\cdot 10^{-6} $ \\
     $8.7 $ & $3$ & $6$ & $1.3\cdot 10^{-6} $ \\
     \hline
     $8.7 $ & $3$ & $7$ & $0.16 $ \\
     $8.8 $ & $3$ & $7$ & $1.1\cdot 10^{-6} $ \\
     $8.9 $ & $3$ & $7$ & $9 \cdot 10^{-7} $ \\
     \hline
     $8.9 $ & $3$ & $8$ & $9.1\cdot 10^{-3} $ \\
     $9 $ & $3$ & $8$ & $7.4\cdot 10^{-7} $ \\
     $9.1 $ & $3$ & $8$ & $6.1\cdot 10^{-7} $ \\
     \hline
     $9.1 $ & $3$ & $9$ & $1.9\cdot 10^{-5} $ \\
     $9.2 $ & $3$ & $9$ & $5.1\cdot 10^{-7} $ \\
     $9.3 $ & $3$ & $9$ & $4.2\cdot 10^{-7} $ \\
     \hline
     $ 9.2$ & $3$ & $10$ & $2.73 $ \\
     $ 9.3$ & $3$ & $10$ & $4.2\cdot 10^{-7} $ \\
     $ 9.4$ & $3$ & $10$ & $ 3.5\cdot 10^{-7}$ \\
     \hline
    \end{tabular}
    \quad
     \begin{tabular}{ | l | l | l | l | }
    \hline
     $Y_0$ & $\alpha_1$ & $\alpha_2$ & $C $ \\  
    \hline
     $7.6 $ & $4$ & $1$ & $4.9\cdot 10^{-2}$ \\
     $7.7 $ & $4$ & $1$ & $2.4\cdot 10^{-5} $ \\
     $7.8 $ & $4$ & $1$ & $7.6\cdot 10^{-6} $ \\
     \hline
     $ 7.9$ & $4$ & $2$ & $2.2\cdot 10^{-2} $ \\
     $ 8$ & $4$ & $2$ & $6.3\cdot 10^{-6} $ \\
     $ 8.1$ & $4$ & $2$ & $4.3\cdot 10^{-6} $ \\
     \hline
     $ 8.1$ & $4$ & $3$ & $9.5 $ \\
     $ 8.2$ & $4$ & $3$ & $1.8\cdot 10^{-4} $ \\
     $ 8.3$ & $4$ & $3$ & $2.9\cdot 10^{-6} $ \\
     \hline
     $ 8.4$ & $4$ & $4$ & $4.7\cdot 10^{-3} $ \\
     $ 8.5$ & $4$ & $4$ & $2\cdot 10^{-6} $ \\
     $ 8.6$ & $4$ & $4$ & $1.7\cdot 10^{-6} $ \\
     \hline
     $ 8.6$ & $4$ & $5$ & $1.7\cdot 10^{-2} $ \\
     $ 8.7$ & $4$ & $5$ & $1.4\cdot 10^{-6} $ \\
     $ 8.8$ & $4$ & $5$ & $1.1\cdot 10^{-6} $ \\
     \hline
     $ 8.8$ & $4$ & $6$ & $4.8\cdot 10^{-3} $ \\
     $8.9 $ & $4$ & $6$ & $9.1\cdot 10^{-7} $ \\
     $9 $ & $4$ & $6$ & $7.4 \cdot 10^{-7}$ \\
     \hline
     $ 9$ & $4$ & $7$ & $6.9\cdot 10^{-5} $ \\
     $ 9.1$ & $4$ & $7$ & $6.2\cdot 10^{-7} $ \\
     $ 9.2$ & $4$ & $7$ & $5.1\cdot 10^{-7} $ \\
     \hline
     $ 9.1$ & $4$ & $8$ & $15.59 $ \\
     $ 9.2$ & $4$ & $8$ & $5.4\cdot 10^{-7} $ \\
     $ 9.3$ & $4$ & $8$ & $4.2\cdot 10^{-7} $ \\
     \hline
     $ 9.3$ & $4$ & $9$ & $1.6\cdot 10^{-3} $ \\
     $ 9.4$ & $4$ & $9$ & $3.5\cdot 10^{-7} $ \\
     $ 9.5$ & $4$ & $9$ & $2.9\cdot 10^{-7} $ \\
     \hline
     $ 9.4$ & $4$ & $10$ & $42.5 $ \\
     $ 9.5$ & $4$ & $10$ & $2.9\cdot 10^{-7} $ \\
     $ 9.6$ & $4$ & $10$ & $2.4\cdot 10^{-7} $ \\
     \hline
    \end{tabular}
    \quad 
\end{table}
\begin{table}[H] 
    \small
    \begin{tabular}{ | l | l | l | l |  }
    \hline
     $Y_0$ & $\alpha_1$ & $\alpha_2$ & $C  $ \\  
    \hline
     $ 8$ & $5$ & $1$ & $ 1.9\cdot 10^{-2} $ \\
     $ 8.1$ & $5$ & $1$ & $5.1 \cdot 10^{-6} $ \\
     $ 8.2$ & $5$ & $1$ & $3.5\cdot 10^{-6} $ \\
     \hline
     $ 8.2$ & $5$ & $2$ & $3.87 $ \\
     $ 8.3$ & $5$ & $2$ & $5.6\cdot 10^{-5} $ \\
     $ 8.4$ & $5$ & $2$ & $2.4\cdot 10^{-6} $ \\
     \hline
     $ 8.5$ & $5$ & $3$ & $5.8\cdot 10^{-4} $ \\
     $ 8.6$ & $5$ & $3$ & $1.7 \cdot 10^{-6}$ \\
     $ 8.7$ & $5$ & $3$ & $1.4 \cdot 10^{-6}$ \\
     \hline
     $ 8.7$ & $5$ & $4$ & $6.9\cdot 10^{-4} $ \\
     $ 8.8$ & $5$ & $4$ & $1.1 \cdot 10^{-6}$ \\
     $ 8.9$ & $5$ & $4$ & $9.1\cdot 10^{-7} $ \\
     \hline
     $ 8.9$ & $5$ & $5$ & $5.7\cdot 10^{-5} $ \\
     $ 9$ & $5$ & $5$ & $7.5 \cdot 10^{-7}$ \\
     $ 9.1$ & $5$ & $5$ & $6.2 \cdot 10^{-7}$ \\
     \hline
     $ 9$ & $5$ & $6$ & $18.27 $ \\
     $ 9.1$ & $5$ & $6$ & $8\cdot 10^{-7} $ \\
     $ 9.2$ & $5$ & $6$ & $5.1\cdot 10^{-7} $ \\
     \hline
     $ 9.2$ & $5$ & $7$ & $1.9\cdot 10^{-2} $ \\
     $ 9.3$ & $5$ & $7$ & $ 4.2\cdot 10^{-7}$ \\
     $ 9.4$ & $5$ & $7$ & $ 3.5\cdot 10^{-7}$ \\
     \hline
     $ 9.3$ & $5$ & $8$ & $924 $ \\
     $ 9.4$ & $5$ & $8$ & $6.2\cdot 10^{-7} $ \\
     $ 9.5$ & $5$ & $8$ & $2.9\cdot 10^{-7} $ \\
     \hline
     $ 9.5$ & $5$ & $9$ & $2.8\cdot 10^{-3} $ \\
     $ 9.6$ & $5$ & $9$ & $2.4\cdot 10^{-7} $ \\
     $ 9.7$ & $5$ & $9$ & $2\cdot 10^{-7} $ \\
     \hline
     $ 9.6$ & $5$ & $10$ & $11.93 $ \\
     $ 9.7$ & $5$ & $10$ & $1.6\cdot 10^{-7} $ \\
     $ 9.8$ & $5$ & $10$ & $1.4\cdot 10^{-7} $ \\
     \hline
    \end{tabular}
    \quad
    \begin{tabular}{ | l | l | l | l |  }
    \hline
     $Y_0$ & $\alpha_1$ & $\alpha_2$ & $C  $ \\  
    \hline
     $ 8.3$ & $6$ & $1$ & $0.78 $ \\
     $ 8.4$ & $6$ & $1$ & $1\cdot 10^{-6}$ \\
     $ 8.5$ & $6$ & $1$ & $2\cdot 10^{-6} $ \\
     \hline
     $8.5 $ & $6$ & $2$ & $13.9 $ \\
     $ 8.6$ & $6$ & $2$ & $3.3 \cdot 10^{-5}$ \\
     $8.7 $ & $6$ & $2$ & $1.4 \cdot 10^{-6}$ \\
     \hline
     $8.7 $ & $6$ & $3$ & $30.58 $ \\
     $ 8.8$ & $6$ & $3$ & $1.4\cdot 10^{-5} $ \\
     $ 8.9$ & $6$ & $3$ & $1\cdot 10^{-6} $ \\
     \hline
     $8.9 $ & $6$ & $4$ & $5.1 $ \\
     $9 $ & $6$ & $4$ & $1.1\cdot 10^{-6} $ \\
     $9.1 $ & $6$ & $4$ & $6.1\cdot 10^{-7} $ \\
     \hline
     $9.1 $ & $6$ & $5$ & $4\cdot 10^{-2} $ \\
     $9.2 $ & $6$ & $5$ & $5.1\cdot 10^{-7} $ \\
     $9.3 $ & $6$ & $5$ & $4.2\cdot 10^{-7} $ \\
     \hline
     $9.2 $ & $6$ & $6$ & $3471 $ \\
     $ 9.3$ & $6$ & $6$ & $7.7\cdot 10^{-6} $ \\
     $ 9.4$ & $6$ & $6$ & $3.5\cdot 10^{-7} $ \\
     \hline
     $ 9.4$ & $6$ & $7$ & $0.16 $ \\
     $9.5 $ & $6$ & $7$ & $2.9\cdot 10^{-7} $ \\
     $9.6 $ & $6$ & $7$ & $2.4\cdot 10^{-7} $ \\
     \hline
     $ 9.5$ & $6$ & $8$ & $1498 $ \\
     $9.6 $ & $6$ & $8$ & $2.9 \cdot 10^{-7} $ \\
     $9.7 $ & $6$ & $8$ & $2\cdot 10^{-7} $ \\
     \hline
     $9.7 $ & $6$ & $9$ & $7.4\cdot 10^{-5} $ \\
     $9.8 $ & $6$ & $9$ & $1.6\cdot 10^{-7} $ \\
     $9.9 $ & $6$ & $9$ & $1.4\cdot 10^{-7} $ \\
     \hline
     $9.8 $ & $6$ & $10$ & $3.9\cdot 10^{-2} $ \\
     $9.9 $ & $6$ & $10$ & $1.4\cdot 10^{-7} $ \\
     $10 $ & $6$ & $10$ & $1.1\cdot 10^{-7} $ \\
     \hline
    \end{tabular}
    \quad
     \begin{tabular}{ | l | l | l | l | }
    \hline
     $Y_0$ & $\alpha_1$ & $\alpha_2$ & $C $ \\  
    \hline
     $8.6 $ & $7$ & $1$ & $0.57$ \\
     $8.7 $ & $7$ & $1$ & $2.2\cdot 10^{-6} $ \\
     $8.8 $ & $7$ & $1$ & $1.2\cdot 10^{-6} $ \\
     \hline
     $8.8 $ & $7$ & $2$ & $0.43 $ \\
     $8.9 $ & $7$ & $2$ & $1.1\cdot 10^{-6} $ \\
     $9 $ & $7$ & $2$ & $7.5\cdot 10^{-7} $ \\
     \hline
     $9 $ & $7$ & $3$ & $2\cdot 10^{-2} $ \\
     $9.1 $ & $7$ & $3$ & $6.2\cdot 10^{-7} $ \\
     $9.2 $ & $7$ & $3$ & $5.1\cdot 10^{-7} $ \\
     \hline
     $9.2 $ & $7$ & $4$ & $3.4\cdot 10^{-5} $ \\
     $9.3 $ & $7$ & $4$ & $4.2\cdot 10^{-7} $ \\
     $9.4 $ & $7$ & $4$ & $3.5\cdot 10^{-7} $ \\
     \hline
     $9.3 $ & $7$ & $5$ & $1.3 $ \\
     $9.4 $ & $7$ & $5$ & $3.5\cdot 10^{-7} $ \\
     $9.5 $ & $7$ & $5$ & $2.9\cdot 10^{-7} $ \\
     \hline
     $9.5 $ & $7$ & $6$ & $8.2\cdot 10^{-6} $ \\
     $9.6 $ & $7$ & $6$ & $2.4\cdot 10^{-7} $ \\
     $9.7 $ & $7$ & $6$ & $2\cdot 10^{-7} $ \\
     \hline
     $9.6 $ & $7$ & $7$ & $2.9\cdot 10^{-2} $ \\
     $9.7 $ & $7$ & $7$ & $2\cdot 10^{-7} $ \\
     $9.8 $ & $7$ & $7$ & $1.6\cdot 10^{-7} $ \\
     \hline
     $9.7 $ & $7$ & $8$ & $44.51 $ \\
     $9.8 $ & $7$ & $8$ & $1.6\cdot 10^{-7} $ \\
     $9.9 $ & $7$ & $8$ & $1.4\cdot 10^{-7} $ \\
     \hline
     $9.9 $ & $7$ & $9$ & $1.5\cdot 10^{-7} $ \\
     $10 $ & $7$ & $9$ & $1.1\cdot 10^{-7} $ \\
     $10.1 $ & $7$ & $9$ & $8.9\cdot 10^{-8} $ \\
     \hline
     $10 $ & $7$ & $10$ & $9.1\cdot 10^{-7} $ \\
     $10.1 $ & $7$ & $10$ & $9\cdot 10^{-8} $ \\
     $10.2 $ & $7$ & $10$ & $7.4\cdot 10^{-8} $ \\
     \hline
    \end{tabular}
    \quad 
\end{table}

\begin{table}[H] 
    \small
    \begin{tabular}{ | l | l | l | l |  }
    \hline
     $Y_0$ & $\alpha_1$ & $\alpha_2$ & $C  $ \\  
    \hline
     $8.9 $ & $8$ & $1$ & $2.7\cdot 10^{-3} $ \\
     $9 $ & $8$ & $1$ & $7.6\cdot 10^{-7} $ \\
     $9.1 $ & $8$ & $1$ & $6.2\cdot 10^{-7} $ \\
     \hline
     $9.1 $ & $8$ & $2$ & $3.3\cdot 10^{-5} $ \\
     $9.2 $ & $8$ & $2$ & $5.2\cdot 10^{-7} $ \\
     $9.3 $ & $8$ & $2$ & $4.2\cdot 10^{-7} $ \\
     \hline
     $9.2 $ & $8$ & $3$ & $2.11 $ \\
     $9.3 $ & $8$ & $3$ & $4.4\cdot 10^{-7} $ \\
     $9.4 $ & $8$ & $3$ & $3.5\cdot 10^{-7} $ \\
     \hline
     $ 9.4$ & $8$ & $4$ & $1.7\cdot 10^{-4} $ \\
     $ 9.5$ & $8$ & $4$ & $2.9\cdot 10^{-7} $ \\
     $ 9.6$ & $8$ & $4$ & $ 2.4\cdot 10^{-7}$ \\
     \hline
     $ 9.5$ & $8$ & $5$ & $1.43 $ \\
     $9.6 $ & $8$ & $5$ & $2.4\cdot 10^{-7} $ \\
     $9.7 $ & $8$ & $5$ & $2\cdot 10^{-7} $ \\
     \hline
     $9.7 $ & $8$ & $6$ & $3.6\cdot 10^{-7} $ \\
     $9.8 $ & $8$ & $6$ & $1.6\cdot 10^{-7} $ \\
     $9.9 $ & $8$ & $6$ & $1.4\cdot 10^{-7} $ \\
     \hline
     $9.8 $ & $8$ & $7$ & $7.9\cdot 10^{-5} $ \\
     $9.9 $ & $8$ & $7$ & $1.4\cdot 10^{-7} $ \\
     $10 $ & $8$ & $7$ & $1.1\cdot 10^{-7} $ \\
     \hline
     $ 9.9$ & $8$ & $8$ & $1.4\cdot 10^{-2} $ \\
     $10 $ & $8$ & $8$ & $1.1\cdot 10^{-7} $ \\
     $10.1 $ & $8$ & $8$ & $9\cdot 10^{-8} $ \\
     \hline
     $10 $ & $8$ & $9$ & $0.82 $ \\
     $10.1 $ & $8$ & $9$ & $9\cdot 10^{-8} $ \\
     $10.2 $ & $8$ & $9$ & $7.4\cdot 10^{-8} $ \\
     \hline
     $10.1 $ & $8$ & $10$ & $14.03 $ \\
     $10.2 $ & $8$ & $10$ & $7.4\cdot 10^{-8} $ \\
     $10.3 $ & $8$ & $10$ & $6.1\cdot 10^{-8} $ \\
     \hline
    \end{tabular}
    \quad
    \begin{tabular}{ | l | l | l | l |  }
    \hline
     $Y_0$ & $\alpha_1$ & $\alpha_2$ & $C  $ \\  
    \hline
     $9.1 $ & $9$ & $1$ & $0.78 $ \\
     $9.2 $ & $9$ & $1$ & $5.4\cdot 10^{-7} $ \\
     $9.3 $ & $9$ & $1$ & $4.3\cdot 10^{-7} $ \\
     \hline
     $9.3 $ & $9$ & $2$ & $6.3\cdot 10^{-4} $ \\
     $9.4 $ & $9$ & $2$ & $3.5\cdot 10^{-7} $ \\
     $9.5 $ & $9$ & $2$ & $2.9\cdot 10^{-7} $ \\
     \hline
     $9.4 $ & $9$ & $3$ & $10.85 $ \\
     $9.5 $ & $9$ & $3$ & $3\cdot 10^{-7} $ \\
     $9.6 $ & $9$ & $3$ & $2.4\cdot 10^{-7} $ \\
     \hline
     $9.6 $ & $9$ & $4$ & $2.5\cdot 10^{-5} $ \\
     $9.7 $ & $9$ & $4$ & $2\cdot 10^{-7} $ \\
     $9.8 $ & $9$ & $4$ & $1.6\cdot 10^{-7} $ \\
     \hline
     $ 9.7$ & $9$ & $5$ & $3.5\cdot 10^{-2} $ \\
     $9.8 $ & $9$ & $5$ & $1.7\cdot 10^{-7} $ \\
     $9.9 $ & $9$ & $5$ & $1.4\cdot 10^{-7} $ \\
     \hline
     $9.8 $ & $9$ & $6$ & $19.97 $ \\
     $9.9 $ & $9$ & $6$ & $1.4\cdot 10^{-7} $ \\
     $10 $ & $9$ & $6$ & $ 1.1\cdot 10^{-7}$ \\
     \hline
     $10 $ & $9$ & $7$ & $1.1\cdot 10^{-7} $ \\
     $10.1 $ & $9$ & $7$ & $9\cdot 10^{-8} $ \\
     $10.2 $ & $9$ & $7$ & $7.4\cdot 10^{-8} $ \\
     \hline
     $10.1 $ & $9$ & $8$ & $1.2\cdot 10^{-7} $ \\
     $10.2 $ & $9$ & $8$ & $7.4\cdot 10^{-8} $ \\
     $10.3 $ & $9$ & $8$ & $6.2\cdot 10^{-8} $ \\
     \hline
     $10.2 $ & $9$ & $9$ & $1.6\cdot 10^{-7} $ \\
     $10.3 $ & $9$ & $9$ & $6.1\cdot 10^{-8} $ \\
     $10.4 $ & $9$ & $9$ & $5\cdot 10^{-8} $ \\
     \hline
     $10.3 $ & $9$ & $10$ & $1.4\cdot 10^{-7} $ \\
     $10.4 $ & $9$ & $10$ & $5.1\cdot 10^{-8} $ \\
     $10.5 $ & $9$ & $10$ & $4.2\cdot 10^{-8} $ \\
     \hline
    \end{tabular}
    \quad
    \begin{tabular}{ | l | l | l | l | }
    \hline
     $Y_0$ & $\alpha_1$ & $\alpha_2$ & $C $ \\  
    \hline
     $ 9.3$ & $10$ & $1$ & $15.61 $ \\
     $9.4 $ & $10$ & $1$ & $4.2\cdot 10^{-7} $ \\
     $9.5 $ & $10$ & $1$ & $2.9\cdot 10^{-7} $ \\
     \hline
     $9.5 $ & $10$ & $2$ & $5.2\cdot 10^{-4} $ \\
     $9.6 $ & $10$ & $2$ & $2.4\cdot 10^{-7} $ \\
     $9.7 $ & $10$ & $2$ & $2\cdot 10^{-7} $ \\
     \hline
     $9.6 $ & $10$ & $3$ & $1.84 $ \\
     $9.7 $ & $10$ & $3$ & $2\cdot 10^{-7} $ \\
     $9.8 $ & $10$ & $3$ & $1.7\cdot 10^{-7} $ \\
     \hline
     $9.8 $ & $10$ & $4$ & $2.2\cdot 10^{-7} $ \\
     $9.9 $ & $10$ & $4$ & $1.4\cdot 10^{-7} $ \\
     $10 $ & $10$ & $4$ & $1.1\cdot 10^{-7} $ \\
     \hline
     $9.9 $ & $10$ & $5$ & $1.1\cdot 10^{-5} $ \\
     $10 $ & $10$ & $5$ & $1.1\cdot 10^{-7} $ \\
     $10.1 $ & $10$ & $5$ & $9 \cdot 10^{-8} $ \\
     \hline
     $10 $ & $10$ & $6$ & $6.3\cdot 10^{-4} $ \\
     $10.1 $ & $10$ & $6$ & $9\cdot 10^{-7} $ \\
     $10.2 $ & $10$ & $6$ & $7.4\cdot 10^{-7} $ \\
     \hline
     $10.1 $ & $10$ & $7$ & $1.2\cdot 10^{-2} $ \\
     $10.2 $ & $10$ & $7$ & $7.4\cdot 10^{-8} $ \\
     $10.3 $ & $10$ & $7$ & $6.1\cdot 10^{-8} $ \\
     \hline
     $10.2 $ & $10$ & $8$ & $6\cdot 10^{-2}  $ \\
     $10.3 $ & $10$ & $8$ & $6.1\cdot 10^{-8}  $ \\
     $10.4 $ & $10$ & $8$ & $5.1\cdot 10^{-8}  $ \\
     \hline
     $10.3 $ & $10$ & $9$ & $7.4\cdot 10^{-2}  $ \\
     $10.4 $ & $10$ & $9$ & $5.1\cdot 10^{-8}  $ \\
     $10.5 $ & $10$ & $9$ & $4.2\cdot 10^{-8}  $ \\
     \hline
     $10.4 $ & $10$ & $10$ & $2\cdot 10^{-2}  $ \\
     $10.5 $ & $10$ & $10$ & $4.2\cdot 10^{-8}  $ \\
     $10.6 $ & $10$ & $10$ & $3.4\cdot 10^{-8}  $ \\
     \hline
    \end{tabular}
    \quad
      
\end{table}

\end{document}